\documentclass{amsart}
\usepackage{amsmath,amsfonts,amsthm,amssymb,stmaryrd,comment,color}

\renewcommand{\leq}{\leqslant}
\renewcommand{\geq}{\geqslant}

\newcommand{\bbz}{\mathbb{Z}}
\newcommand{\bbr}{\mathbb{R}}
\newcommand{\bbq}{\mathbb{Q}}
\newcommand{\bbc}{\mathbb{C}}

\newcommand{\ls}{\lesssim}

\newcommand*{\bbe}{
  \mathop{
    \mathchoice{\vcenter{\hbox{\larger[4]$\mathbb{E}$}}}
               {\kern0pt\mathbb{E}}
               {\kern0pt\mathbb{E}}
               {\kern0pt\mathbb{E}}
  }\displaylimits
}

\newcommand{\abs}[1]{\left\lvert #1\right\rvert}
\newcommand{\Abs}[1]{\lvert #1\rvert}
\newcommand{\brac}[1]{\left( #1\right)}

\newcommand{\norm}[1]{\left\lVert #1\right\rVert}

\newtheorem{theorem}{Theorem}
\newtheorem{lemma}{Lemma}

\newtheorem{corollary}{Corollary}

\newtheorem{proposition}{Proposition}

\theoremstyle{definition}
\newtheorem{definition}{Definition}

\begin{document}
\title[Sum-product problem with few prime factors]{More on the sum-product problem for integers with few prime factors}
\author{Rishika Agrawal}
\address{Department of Mathematics, University of Georgia, Athens, GA, 30602 }
\email{rishika.agrawal@uga.edu}
\author{Thomas F. Bloom}
\address{Department of Mathematics, University of Manchester, Manchester, M13 9PL}
\email{thomas.bloom@manchester.ac.uk}
\author{Giorgis Petridis}
\address{Department of Mathematics, University of Georgia, Athens, GA, 30602 }
\email{giorgis@cantab.net}

\begin{abstract}
We show that if $A\subset \bbz$ is a finite set of integers in which every integer is divisible by $O(1)$ many primes then
\[\max(\lvert A+A\rvert,\lvert AA\rvert) \geq \lvert A\rvert^{12/7-o(1)}\]
and, for any $m\geq 2$,
\[\max(\lvert mA\rvert, \lvert A^{(m)}\rvert) \geq \lvert A\rvert^{\frac{2}{3}m+\frac{1}{3}-o(1)}.\]
Finally, we show that if $A\subset \bbq$ is a finite set of rationals in which the numerator and denominator of every $x\in A$ is divisible by $O(1)$ many primes then $\lvert A+AA\rvert \geq \lvert A\rvert^{2-o(1)}$.
\end{abstract}

\maketitle

The sum-product conjecture, a cornerstone of additive combinatorics, states (in its classical form) that, for any finite set of integers $A\subset \bbz$, 
\[\max(\abs{A+A},\abs{AA})\geq \abs{A}^{2-o(1)}.\]
This is often attributed to Erd\H{o}s and Szemer\'{e}di, who proved the first results in this direction \cite{ErSz83}, but it seems to have first appeared in the literature in a problems list of Erd\H{o}s in 1977 \cite{Er77}. The best-known bounds in this direction for arbitrary sets of integers have an exponent of the shape $\frac{4}{3}+c$ for some small constant $c>0$ - we refer to \cite{Bl25} for more on the history and the latest progress in this direction.

This paper is concerned with the sum-product phenomenon under the assumption that every $n\in A$ has few prime factors. This variant was first studied by Hanson, Rudnev, Shkredov, and Zhelezov \cite{HRSZ25}, who showed that this assumption allows for much stronger bounds, proving the following.
\begin{theorem}[Hanson, Rudnev, Shkredov, and Zhelezov \cite{HRSZ25}]\label{th-old}
Let $A\subset \bbz$ be a finite set of integers and let $k\geq 1$ be a fixed integer. If $\omega(n)\leq k$ for all $n\in A$ then
\[\max(\abs{A+A},\abs{AA})\gg \abs{A}^{5/3-o(1)}.\]
\end{theorem}
The exponent $5/3$ is, in a sense, a natural barrier of the method used in \cite{HRSZ25}. Indeed, they actually prove the stronger result that either $\abs{AA}\gg \abs{A}^{5/3-o(1)}$ or the additive energy\footnote{We review standard definitions and notation in Section~\ref{sec-defs}.} $E(A)$ is at most $\abs{A}^{7/3+o(1)}$ (which by the Cauchy-Schwarz inequality implies $\abs{A+A}\geq \abs{A}^{5/3-o(1)}$). An example of Balog and Wooley \cite{BaWo17} (that we give in Section~\ref{sec-sketch}) shows that the exponent $5/3$ is best possible for this stronger statement with additive energy.

In this paper we nonetheless give an improvement for the original statement, replacing the exponent $5/3$ with $12/7$, pushing past the Balog-Wooley barrier.

\begin{theorem}\label{th-main}
Let $A\subset \bbz$ be a finite set of integers and let $k\geq 1$ be a fixed integer. If $\omega(n)\leq k$ for all $n\in A$ then
\[\max(\abs{A+A},\abs{AA})\gg \abs{A}^{12/7-o(1)}.\]
and
\[\max(\abs{A-A},\abs{AA})\gg \abs{A}^{12/7-o(1)}.\]
\end{theorem}

For brevity we have stated this with $k$ some fixed integer, but (just as in \cite{HRSZ25}) the method works, and delivers the same exponent, as long as $k=o(\frac{\log \abs{A}}{\log\log \abs{A}})$. A precise quantitative statement is given in Theorem~\ref{th-mainprecise}.

Those familiar with \cite{HRSZ25} may also be interested to note that we have found a simplification of their method that allows one to completely avoid the use of Chang's lemma and any kind of martingale technology, which played a crucial role in \cite{HRSZ25}. This is instead replaced with an application of H\"{o}lder's inequality to amplify the consequences of the bounds on $S$-unit equations. We take the opportunity to give this simplified proof of the main result of \cite{HRSZ25} in Section~\ref{sec-add2}.

Another popular problem in the sum-product family asks about the behaviour of iterated sum and product sets $mA=A+\cdots+A$ and $A^{(m)}=A\cdots A$ as $m\to \infty$. A natural generalisation of the original sum-product conjecture (also mentioned in \cite{ErSz83}) is that, for any finite set of integers $A\subset \bbz$ and any $m\geq 1$,
\[\max(\abs{mA},\Abs{A^{(m)}})\geq \abs{A}^{(1-o(1))m}.\]
The first non-trivial results in this direction were achieved by Bourgain and Chang \cite{BoCh04}, who proved that 
\[\max(\abs{mA},\Abs{A^{(m)}})\geq \abs{A}^{f(m)}\]
for some function $f(m)$ which $\to \infty$ as $m\to \infty$. The best-known bounds known for arbitrary finite sets of integers, due to P\'alv\"olgyi and Zhelezov \cite{PaZh21}, allow one to take $f(m)\gg \frac{\log m}{\log\log m}$.

Our second main result in this paper is that, under the same assumption that all $n\in A$ have few prime factors, we can establish such a lower bound with $f(m)\gg m$.
\begin{theorem}\label{th-main2}
Let $A\subset \bbz$ be a finite set of integers and let $k\geq 1$ be a fixed integer. If $\omega(n)\leq k$ for all $n\in A$ then for all $m\geq 1$
\[\max(\abs{mA},\Abs{A^{(m)}})\gg \abs{A}^{\frac{2}{3}m+\frac{1}{3}-o(m)}.\]
\end{theorem}
As above, these results are valid up to $k=o(\frac{\log \abs{A}}{\log\log \abs{A}})$. A full precise quantitative statement is given in Theorem~\ref{th-main2precise}. Note that the case $m=2$ recovers Theorem~\ref{th-old} of Hanson, Rudnev, Shkredov, and Zhelezov above. By incorporating some of the ideas in the proof of Theorem~\ref{th-main} the constant term $1/3$ can be improved slightly.

Our third main result concerns lower bounds for $\abs{A+AA}$. As a single quantity which combines both addition and multiplication, a folklore conjecture is that this should be at least $\abs{A}^{2-o(1)}$ for any finite set $A\subset \bbr$. For sets of integers (or, in general, $1$-separated sets of reals) a simple proof of $\abs{A+AA}\geq \abs{A}^2$ was given by Shakan \cite{Sh}. The best-known lower bound for arbitrary finite $A\subset \bbr$ is $\abs{A+AA}\geq \abs{A}^{3/2+c}$ for some (constant but small) $c>0$, due to Roche-Newton, Ruzsa, Shen, and Shkredov \cite{RRSS19}, with a similar result for the corresponding `energy' given in \cite[Theorem 9]{RuSh22}. The following result gives a bound of comparable quality for finite sets of rationals, under a similar `few primes' assumption, and also characterises nearly extremal sets. Such a `stability result' is not known for arbitrary subsets of integers.

\begin{theorem}\label{th-main3}
Let $A\subset \bbq$ be a finite set of rationals and let $k\geq 1$ be a fixed integer. If $\omega(a)+\omega(b)\leq k$ for all $a/b\in A$ with $(a,b)=1$ then
\[\abs{A+AA}\geq \abs{A}^{2-o(1)}.\]
Furthermore, if $\abs{A+AA}\leq M\abs{A}^{2}$ then
\[\max(\abs{AA},\abs{A+A})\geq M^{-O(1)}\abs{A}^{2-o(1)}.\]
\end{theorem}
As above, these results are valid up to $k=o(\frac{\log \abs{A}}{\log\log \abs{A}})$. A full precise quantitative statement is given in  Corollary~\ref{cor-main3precise}. Some loss of $\abs{A}^{o(1)}$ is necessary here, as Roche-Newton, Ruzsa, Shen, and Shkredov \cite{RRSS19} have constructed a finite $A\subset \bbq$ such that
\[\abs{A+AA}\leq \frac{\abs{A}^2}{(\log \abs{A})^{c}}\]
for some constant $c>0$.

Hanson, Rudnev, Shkredov, and Zhelezov \cite{HRSZ25} have noted that their methods yield a result similar to Theorem~\ref{th-main3} for the set $AA+AA$. One may deduce this from Theorem~\ref{th-main3} because all of $|AA+AA|, |A+A|, |AA|$ are dilation invariant and we may therefore suppose, without loss of generality, that $1 \in A$, whence $A+AA\subseteq AA+AA$.

After reviewing basic definitions, in Section~\ref{sec-sketch} we will sketch the main ideas of the proofs. In Section~\ref{sec-comb} we will transfer the property of having few prime factors to being efficiently covered by a multiplicative group (in a similar fashion as \cite{HRSZ25}) and finally in Sections~\ref{sec-add} and \ref{sec-add2} we will show how being efficiently covered by a multiplicative group leads to strong lower bounds for the sumset.

\subsection*{Acknowledgements} TB is supported by a Royal Society University Research Fellowship. GP is supported by the Simons Foundation grant MPS-TSM-00007816. This material is based upon work supported by the National Science Foundation under Grant No. 2054214. We thank Ilya Shkredov for bringing our attention to some existing work of his on higher energies, which led to an improved final exponent. We would also like to thank Akshat Mudgal for useful feedback on an early draft of this paper and Misha Rudnev for generously sharing his insight on the topic.

\section{Definitions}\label{sec-defs}
All sets considered in this paper will be finite sets of some fixed abelian group (usually either $\bbc$, $\bbq$, or $\bbz$). We write $1_A$ for the indicator function of $A$, and define the convolution and difference convolution by
\[1_A\ast 1_B(x)=\sum_{b\in B}1_A(x-b)\textrm{ and }1_A\circ 1_B(x)=\sum_{b\in B}1_A(x+b).\]
We consider $L^p$ norms with the counting measure, so that for any $p\geq 1$
\[\norm{f}_p^p=\sum_x \lvert f(x)\rvert^p\textrm{ and }\norm{f}_\infty=\max\lvert f(x)\rvert.\]
We will only ever consider functions with finite support, so we are free from worries about convergence. When we work in a finite group $G$, the expectation of a function $f: G \to \bbc$ is taken to be \[\bbe_x f(x) = \frac1{\abs{G}} \sum_{x \in G} f(x).\]
The sum set and product set of $A$ and $B$ are defined by
\[A+B = \{ a+b : a\in A\textrm{ and }b\in B\}\quad\textrm{ and }\quad AB = \{ ab : a\in A\textrm{ and }b\in B\}\]
respectively and similarly 
\[ A+AA = \{ a_1 + a_2a_3 : a_1, a_2, a_3 \in A\}.\]
For any $m\geq 1$ the iterated sum and product sets of $A$ are defined by
\[mA = \{a_1+\cdots+a_m : a_1,\ldots,a_m\in A\}\]
and
\[A^{(m)} = \{a_1\cdots a_m : a_1,\ldots,a_m\in A\}.\]
The additive energy between $A$ and $B$ is defined by
\[E(A,B) = \norm{1_A\circ 1_B}_2^2=\#\{ (a_1,a_2,b_1,b_2)\in A\times A\times B\times B : a_1-a_2=b_1-b_2\}.\]
We write $E(A)=E(A,A)$. The higher additive energies are defined, for any $m\geq 1$, by
\[E_{2m}(A) = \sum_{x}1_A^{(m)}(x)^2 = \#\{ (a_1,\ldots,a_{2m})\in A^{2m}: a_1+\cdots +a_m=a_{m+1}+\cdots+a_{2m}\},\]
where $1_A^{(m)}=1_A\ast \cdots \ast 1_A$ is the $m$-fold iterated convolution. We note here that $E(A)=E_4(A)$ and $E_2(A)=\abs{A}$. The fundamental link between additive energies and sum sets is provided by H\"{o}lder's inequality which implies, for any $m\geq 1$,
\begin{equation}\label{eq-holderenergy}
\Abs{mA}E_{2m}(A)\geq \abs{A}^{2m}.
\end{equation}

For any prime $p$ and integer $n\in \bbz\backslash\{0\}$ we write $\nu_p(n)$ for the $p$-adic valuation of $n$, so that $\nu_p(n)=k$ if $p^k\mid n$ and $p^{k+1}\nmid n$. For any integer $n\in \bbz\backslash\{0\}$ the arithmetic function $\omega(n)$ counts the number of distinct prime divisors of $n$ -- that is, the number of $p$ such that $\nu_p(n)\neq 0$.

It is convenient to extend these definitions to $\bbq\backslash\{0\}$. To that end, for any $x\in \bbq\backslash\{0\}$ and prime $p$, if $x=a/b$ then $\nu_p(x)=\nu_p(a)-\nu_p(b)$, and $\omega(x)$ counts the number of distinct primes appearing in $x$ when written in reduced form -- that is, the number of $p$ such that $\nu_p(x)\neq 0$. Similarly, it is convenient to slightly abuse notation and write $p\mid x$ if $\nu_p(x)\neq 0$. We will write $P(x)$ for the set of such primes, so that, for example, $\omega(x)=\abs{P(x)}$. 

If $S$ is any finite set of primes then $\bbq_S\subset \bbq$ is the set of non-zero rationals in which the only primes are those from $S$ -- that is,
\[\bbq_S = \left\{x\in \bbq\backslash\{0\} : p\mid x\implies p\in S\right\}.\]

The rank of a multiplicative group $\Gamma\leq \bbc^\times$ is the smallest $r$ such that there exist $\gamma_1,\ldots,\gamma_r\in \Gamma$ which generate $\Gamma$. Note in particular that $\bbq_S$ is a multiplicative subgroup of $\bbq^\times$ with rank $\abs{S}$.

The general theme of this paper is that sets which are efficiently covered by a bounded rank multiplicative group lack additive structure. To make this precise, the following definition is convenient.

\begin{definition}
A finite set $A\subset \bbc^\times$ is $M$-covered by a rank $r$ multiplicative group if there exists a multiplicative group $\Gamma\subset \bbc^\times$ of rank $r$ and a finite set $B$ of size $\abs{B}\leq M$ such that $A\subseteq \Gamma \cdot B$.
\end{definition}

\section{Sketch of the proofs}\label{sec-sketch}

In this section we will give a sketch of the proofs of our main results. For simplicity, we will work with a fixed set of integers $A$ in which $\omega(n) \ll 1$ for every $n\in A$. Let $M=\abs{AA}/\abs{A}$ and $K=\abs{A+A}/\abs{A}$. For the purposes of this sketch we will not bother to record factors of the shape $\abs{A}^{o(1)}$, including them in the $\ll$ notation.

The argument of \cite{HRSZ25} can be summarised as follows.
\begin{enumerate}
\item By a combinatorial argument there exists a set of primes $S$ with $\abs{S}\ll 1$ and a large subset $A'\subseteq A$ such that $A'$ is contained in $O(M)$ many dilates of $\bbq_S$, say $A'\subseteq \bbq_S\cdot B$. Note that in particular we can write $A'$ as the disjoint union of dilates $c\cdot B_c$ for some $c\in \bbq_S$ and $B_c\subseteq B$.
\item Generalising a lemma due to Chang \cite{Ch03}, using ideas from martingale theory, they prove that
\[E(A) \ll \sum_{c_1,c_2\in \bbq_S}E(c_1B_{c_1},c_2B_{c_2}).\]
\item Importing a deep quantitative bound from the theory of $S$-unit equations they prove that the right-hand side is
\[\ll \abs{A}^2+\abs{A}\abs{B}^2.\]
\item It follows that
\[E(A) \ll \abs{A}^2+\abs{A}M^2,\]
which implies (via $E(A) \geq \abs{A}^3/K$) the bound $\max(K,M)\gg \abs{A}^{2/3}$.
\end{enumerate}

The first part of our argument is the same as that in \cite{HRSZ25}: we pass to a large subset $A'\subseteq A$ which is contained in $O(M)$ many dilates of $\bbq_S$, which is a rank $O(1)$-multiplicative group. This part of the argument is essentially unchanged from \cite{HRSZ25}, although we take the opportunity to streamline the presentation.

Rather than then taking a detour through martingale theory however, we will seek to apply the bounds from $S$-unit theory directly. Roughly speaking, these imply that, if $\Gamma$ is a fixed multiplicative group of rank $r=O(1)$ and $a_0,\ldots,a_m\in \bbc^\times$ are fixed, the number of solutions to
\[a_1z_1+\cdots+a_mz_m=a_0\]
which are non-degenerate (in that no subsum on the left-hand side vanishes) is $O_{m,r}(1)$. 

Applying this bound with $m=3$ and taking the sum of such bounds over all $M$ relevant dilates of $\bbq_S$ this immediately implies that
\[E(A') \ll \abs{A}^2+\abs{A}M^3.\]
This alone only implies $\max(K,M) \gg \abs{A}^{1/2}$. We can amplify this bound, however, using the fact that, by H\"{o}lder's inequality,
\[E(A') \leq \abs{A}E_{2m}(A')^{\frac{1}{m-1}}\]
as $m\to \infty$. The $S$-unit bound yields $E_{2m}(A')\ll_m \abs{A}^m+\abs{A}M^{2m-1}$, whence we immediately get the improved bound
\[E(A') \ll \abs{A}^2+\abs{A}M^{\frac{2m-1}{m-1}}=\abs{A}^2+\abs{A}M^{2+o(1)}.\]
This recovers the energy bound of \cite{HRSZ25} cited above, but without the detour through a Chang-type lemma.

We now turn to the proofs of our new results, Theorems~\ref{th-main}, \ref{th-main2}, and \ref{th-main3}. We first discuss Theorem~\ref{th-main2}. Suppose that $m$ is large and $\Abs{A^{(m)}}\leq \abs{A}^{\frac{2}{3}m+\frac{1}{3}}$. By the pigeonhole principle there exists some $1\leq i<m$ such that if $B=A^{(i)}$ then $\abs{AB}\leq \abs{A}^{\frac{2}{3}}\abs{B}$. Performing an asymmetric generalisation of the combinatorial decomposition mentioned above we may use this to find some large subset $A'\subseteq A$ which is contained in $O(\abs{A}^{2/3})$ many dilates of a multiplicative group of rank $O(1)$. 

As above, the bounds from the theory of $S$-unit equations coupled with H\"{o}lder's inequality yield that
\[E_{2m}(A') \ll \abs{A}^m+\abs{A}(\abs{A}^{2/3})^{2m-2}\ll\abs{A}^{\frac{4}{3}m-\frac{1}{3}},\]
whence
\[\abs{mA}\geq \abs{mA'}\gg \abs{A}^{2m-\frac{4}{3}m+\frac{1}{3}}=\abs{A}^{\frac{2}{3}m+\frac{1}{3}}\]
as required.

Next we sketch the proof of Theorem~\ref{th-main3}, the lower bound on $\abs{A+AA}$. By applying H\"older's inequality we obtain that, for all large $m$,
\[ E(A', AA) \le E_{2m}(A)^{1/m} |AA|^{1+o(1)}\]
for a suitable subset $A' \subseteq A$ of size $\abs{A}^{1-o(1)}$ as above. An application of the aforementioned bounds on $E_{2m}(A')$ implies
\[ E(A', AA) \le (\abs{A} + M^2) |AA|^{1+o(1)}.\] An application of the Cauchy-Schwarz inequality yields
\[ |A'+AA| \ge \min\left\{ \abs{A} \abs{AA}^{1-o(1)} , \frac{\abs{A}^4}{\abs{AA}^{1+o(1)}}\right\}.\]
This immediately implies 
\[ \abs{A+AA} \ge \abs{A'+AA} \ge \abs{A}^{2-o(1)},\]
with near equality either when $\abs{AA} = \abs{A}^{1+o(1)}$ or when $\abs{AA} = \abs{A}^{2-o(1)}$. To complete the characterization of nearly extremal sets, we note that in the former case the bound on $E(A)$ from \cite{HRSZ25} immediately implies that $\abs{A'+A'} = \abs{A}^{2-o(1)}$.

Finally, we turn to the headline result of how to obtain an exponent for 
\[\max(\abs{A+A},\abs{AA})\]
larger than $5/3$. This is a more delicate matter, since as mentioned earlier the inequality
\[\max(\abs{AA},\abs{A}^4E(A')^{-1})\gg \abs{A}^{5/3}\]
for some large $A'\subseteq A$ proved by the above method of \cite{HRSZ25} cannot be improved. This follows by the following example of Balog and Wooley \cite{BaWo17}.

Let $P=\{1,\ldots,M\}$ and $G=\{M,M^2,\ldots,M^N\}$ and let $A=PG$. It is easy to check (certainly if $M$ is a large prime for example) that $\abs{A}\approx MN$. On one hand,
\[\abs{AA}\leq \abs{P}^2\abs{GG}\ll M^2N\ll M\abs{A}.\]
On the other hand, if $A'\subseteq A$ is large then $A'$ must densely intersect many dilates of $P$ -- that is, $\abs{A'\cap (gP)}\gg M$ for $\gg N$ many $g\in G$. Note that
\[E(A'\cap (gP))\gg \frac{M^4}{\abs{gP+gP}}\gg M^3\]
and hence, summing this over $\gg N$ many $g\in G$, we have $E(A')\gg NM^3\gg \abs{A}^4/N^3M$. Choosing $M\approx\abs{A}^{2/3}$ and $N\approx \abs{A}^{1/3}$ we see that, for any $A'\subseteq A$ with $\abs{A'}\gg \abs{A}$,
\[\max(\abs{AA}, \abs{A}^4E(A')^{-1})\ll \abs{A}^{5/3}\]
and hence the bound of \cite{HRSZ25} is sharp.

To go beyond the exponent of $5/3$ we therefore need to find a method that bounds, not the additive energy, but the sumset directly. This is provided by the toolkit of higher additive energies (which has been extensively developed by Shkredov in particular in a series of papers. The three ingredients required are:
\begin{enumerate}
\item a good upper bound for the higher energies $E_{4\ell}(A)$ as $\ell\to \infty$,
\item a good upper bound for $\sum_{x\neq 0}1_A\circ 1_A(x)^4$ (which we accomplish by giving a good pointwise upper bound for $1_A\circ 1_A(x)$ at $x\neq 0$ and using the previous bound for $E(A)$), and 
\item a good lower bound for $K=\abs{A+A}/\abs{A}$ in terms of $E_{4\ell}(A)$ and $\sum_{x\neq 0} 1_A\circ 1_A(x)^4$.
\end{enumerate}

For the third we may use a lemma inspired by similar bounds due to Shkredov (see Lemma~\ref{lem-pop} below), which in particular implies that, for any $\ell \geq 2$,
\begin{equation}\label{eq-sh}
\abs{A}^{9}\leq K^3E_{4\ell}(A)^{\frac{1}{\ell-1}}\brac{K^3\abs{A}^{-1}E_{4\ell}(A)^{\frac{1}{\ell-1}}+\sum_{x\neq 0}1_A\circ 1_A(x)^4}.
\end{equation}
For the first we note that the $S$-unit equation bound coupled with H\"{o}lder's inequality yields, as above,
\[E_{4\ell}(A)^{\frac{1}{\ell-1}} \ll \abs{A}^{2+O(1/\ell)}+\abs{A}^{O(1/\ell)}M^4.\]
Finally, to bound $1_A\circ 1_A(x)$ for fixed $x\neq 0$ we will also use the $S$-unit equation bound. A naive direct application of this bound implies that, if $A$ is contain in $M$ many dilates of a multiplicative group of rank $O(1)$, then 
\[1_A\circ 1_A(x)\ll M^{2}\]
for any $x\neq 0$. This is too weak for our purposes, and hence (as with our bounds for additive energy) we need to amplify the naive argument for employing the $S$-unit equation bound, this time by employing a graph theoretic argument that is a generalisation of an argument of Roche-Newton and Zhelezov \cite{RoZh15}. This yields
\[1_A\circ 1_A(x)\ll M\]
for any $x\neq 0$, and hence, coupled with our bounds on $E(A)$ above,
\[\sum_{x\neq 0}1_A\circ 1_A(x)^4 \ll M^2\brac{\abs{A}^2+\abs{A}M^2}\ll M^2\abs{A}^2+M^4\abs{A}.\]
Employing these bounds in \eqref{eq-sh} above yields, if $L=\max(K,M)$,
\[\abs{A}^9 \ll L^{14}\abs{A}^{-1}+ L^{11}\abs{A}+L^9\abs{A}^2+L^7\abs{A}^3+L^5\abs{A}^4\]
and hence
\[\max(K,M)=L\gg \abs{A}^{5/7}\]
whence
\[\max(\abs{A+A},\abs{AA})\gg \abs{A}^{12/7}\]
as desired.
\section{Finding a subset with small multiplicative dimension}\label{sec-comb}
Following the proof of \cite{HRSZ25}, the first stage is to show that $\omega(n)\ll 1$ for all $n\in A$ implies that $A$ is contained in few cosets of a multiplicative group of bounded rank. In this section we give a variant of their argument suitable for our purposes -- the essential ideas of the proofs are the same, but we take the opportunity to streamline the proofs and give an asymmetric version that will be required for our application to many sums and products. In fact, for possible future applications, we will give two paths (both present in \cite{HRSZ25}), with different quantitative strengths. 

The first step is to show that if $\omega(n)\ll 1$ for all $n\in A$ then there is a set $S$ of $O(1)$ many primes such that a large proportion of pairs $(a_1,a_2)\in A$ have only primes from $S$ in common.

\begin{lemma}\label{lem-warmup1}
Let $A,B\subset \bbq\backslash\{0\}$ be finite sets such that $\omega(n)\leq k$ for all $n\in A$ and $\omega(n)\leq \ell$ for all $n\in B$. The following both hold:
\begin{enumerate}
\item There is a set of primes $S$ with $\abs{S}\leq 2k\ell$ such that
\[\sum_{a\in A}\sum_{b\in B}1_{P(a)\cap P(b)\subseteq S}\geq \frac{1}{2}\abs{A}\abs{B}.\]
\item There is a set of primes $S$ with $\abs{S}\leq k$ and $A'\subseteq A$ of size 
\[\abs{A'}\geq (2\ell)^{-k}\abs{A}\]
such that
\[\sum_{a\in A'}\sum_{b\in B}1_{P(a)\cap P(b)\subseteq S}\geq \frac{1}{2}\abs{A'}\abs{B}.\]
\end{enumerate}
\end{lemma}
\begin{proof}
For the first part, let $S$ be the set of all primes $p$ such that $\sum_{n\in A}1_{p\mid n}\geq \abs{A}/2\ell$. On one hand,
\[\frac{\abs{A}}{2\ell}\abs{S}\leq \sum_{p\in S}\sum_{n\in A}1_{p\mid n}= \sum_{n\in A}\sum_{p\in S} 1_{p\mid n}\leq k\abs{A},\]
and hence $\abs{S}\leq 2k\ell$. On the other hand, for any prime $p\not\in S$,
\[\sum_{a\in A}\sum_{b\in B}1_{p\in P(a)\cap P(b)}=\brac{\sum_{a\in A}1_{p\mid a}}\brac{\sum_{b\in B}1_{p\mid b}}<\frac{\abs{A}}{2\ell}\sum_{b\in B}1_{p\mid b}.\]
It follows that
\[\sum_{a\in A}\sum_{b\in B}1_{P(a)\cap P(b)\not\subseteq S}\leq \sum_{p\not\in S}\sum_{a\in A}\sum_{b\in B}1_{p\in P(a)\cap P(b)}
<\frac{\abs{A}}{2\ell}\sum_{b\in B}\sum_{p\not\in S}1_{p\mid b}\leq \tfrac{1}{2}\abs{A}\abs{B},\]
and the conclusion follows.

For the second part, let $S$ be a maximal set of primes such that, with $\abs{S}=r$, there is a subset $A'\subseteq A$ of size $\abs{A'}\geq \abs{A}/(2\ell)^r$ with $p\mid n$ for all $n\in A'$ and $p\in S$. (Note that such a maximal $S$ certainly exists, since $S=\emptyset$ satisfies the conditions and $\abs{S}\leq k$ for any such $S$ by assumption.)

By maximality of $S$, for any prime $p\not\in S$,
\[\sum_{a\in A'}\sum_{b\in B}1_{p\in P(a)\cap P(b)}=\brac{\sum_{a\in A'}1_{p\mid a}}\sum_{b\in B}1_{p\mid b}< \frac{\abs{A'}}{2\ell}\sum_{b\in B}1_{p\mid b},\]
and the rest of the proof proceeds as above.
\end{proof}

Note that, in statements such as Lemma~\ref{lem-warmup1}, we allow for the possibility that $S=\emptyset$ (which indeed is necessary in some situations, such as when elements of $A$ and $B$ share no primes in common). 

Secondly, we show that if there are many pairs in $A\times A$ who only share primes from $S$ then either $AA$ is large or $A$ is contained in few cosets of $\bbq_S$. The key idea (from \cite{HRSZ25}) is the observation that if $a$ and $b$ are coprime integers, each with $O(1)$ prime factors, then $a$ and $b$ can be recovered from knowledge of $ab$ at a cost of $O(1)$ -- simply by fixing which primes present in $ab$ belong in $a$ or $b$. This implies, by a simple counting argument, that if there are $\gg \abs{A}^2$ many pairs $a,b\in A$ with $(a,b)=1$ (and $\omega(n)\ll 1$ for all $n\in A$) then $\abs{AA}\gg \abs{A}^2$. The precise argument below is a quantitative form of this, after `factoring out' by $\bbq_S$.

\begin{lemma}\label{lem-warmup2}
Let $A,B\subset \bbq\backslash \{0\}$ be finite sets such that $\omega(n)\leq k$ for all $n\in A$ and $\omega(n)\leq \ell$ for all $n\in B$. Suppose $S$ is a set of primes such that 
\[\sum_{a\in A}\sum_{b\in B}1_{P(a)\cap P(b)\subseteq S}\geq \frac{1}{2}\abs{A}\abs{B}.\]
There is a subset $A'\subseteq A$ of size $\abs{A'}\ge \abs{A}/4$ and a set $C$ such that 
\[\abs{C}\ll 2^{k+\ell}\frac{\abs{AB}}{\abs{B}}\]
and
\[A'\subseteq \bbq_S \cdot C.\]
\end{lemma}
\begin{proof}
Every $x\in \bbq\backslash\{0\}$ can be written uniquely as $x=x_1x_2$ where if $p\mid x_1$ then $p\not\in S$ and $x_2\in \bbq_S$. Separating elements of $A$ according to the value of $x_1$, there exists $C$ such that if $p\mid n\in C$ then $p\not\in S$, and finite $\Gamma_c\subseteq \bbq_S$ for $c\in C$ such that
\[A=\bigcup_{c\in C}(c\cdot \Gamma_c),\]
where the $c\cdot \Gamma_c$ are disjoint. Let $L$ be some parameter to be chosen later, and let $A''\subseteq A$ be the subset of $A$ coming from those $\Gamma_c$ with $\abs{\Gamma_c}<L$. If $\lvert A''\rvert\geq \tfrac{3}{4}\lvert A\rvert$ then there must be at least $\lvert A\rvert\abs{B}/4$ many such pairs $a\in A''$ and $b\in B$ with $P(a)\cap P(b)\subseteq S$.

Any $q\in A''B$ has $<2^{k+\ell}L$ representations as $q=ab$ with $a\in A''$ and $b\in B$ and $P(a)\cap P(b)\subseteq S$. Indeed, writing $a=a_1a_2$ as above (so $a_2\in \bbq_S$ and if $p\mid a_1$ then $p\not\in S$) then one can write $q=a_1a_2b_1b_2$. Since $P(a)\cap P(b)\subseteq S$ there are no primes appearing in both $a_1$ and $a_2b_1b_2$, and hence since $\omega(q)\leq k+\ell$ the prime factorisation of $a_1$ (and hence $a_1$ itself) can be determined from $q$ at a cost of $2^{k+\ell}$. Since $a_2\in \Gamma_{a_1}$ and $a_1a_2\in A''$ the value of $a_2$, and hence the entirety of $a$, can then be determined at a cost of $<L$. 

It follows that 
\[\lvert A\rvert\abs{B}/4 < 2^{k+\ell}L\abs{A''B}\leq 2^{k+\ell}L\abs{AB},\]
which is a contradiction choosing $L=\lvert A\rvert\abs{B}/2^{k+\ell+2}\abs{AB}$. This contradiction means that $\abs{A''}<\frac{3}{4}\abs{A}$, and hence $A'=A\backslash A''$ has size $\abs{A'}\geq \abs{A}/4$. Let $C'\subseteq C$ be the corresponding subset of $C$, so that $A'\subseteq \bbq_S\cdot C'$. By construction
\[L\lvert C'\rvert \leq \lvert A'\rvert,\]
and noting the choice of $L$ above we are done.
\end{proof}

The following proposition follows immediately on combining Lemmas~\ref{lem-warmup1} and \ref{lem-warmup2}.

\begin{proposition}\label{prop-small}
Let $A\subset \bbq\backslash \{0\}$ be a finite set such that $\omega(n)\leq k$ for all $n\in A$. Suppose there is a set $B\subset \bbq\backslash\{0\}$ such that $\omega(n)\leq \ell$ for all $n\in B$ and $\abs{AB}\leq K\abs{B}$. The following both hold:

\begin{enumerate}
\item There is a set $A'\subseteq A$ of size $\abs{A'}\gg \abs{A}$ which is $O(2^{k+\ell}K)$-covered by a rank $O(k\ell)$ multiplicative group.
\item There is a set $A'\subseteq A$ of size $\abs{A'}\gg (2\ell)^{-k}\abs{A}$ which is $O(2^{k+\ell}K)$-covered by a rank $O(k)$ multiplicative group.
\end{enumerate}

\end{proposition}

\section{Additive equations over multiplicative groups}\label{sec-add}

The previous section is the only part of the proof where it is important that we are working over $\bbq$ and use the few prime factors hypothesis. The remainder of the proof concerns the additive structure of sets which are efficiently covered by a low rank multiplicative subgroup of $\bbc^\times$.

There is a deep theory concerning linear equations in multiplicative groups, and we will only require one particular case of this. The following quantitative result was proved by Amoroso and Viada \cite[Theorem 6.2]{AmVi09}, and is a refinement of an earlier quantitative bound due to Evertse, Schlickewei, and Schmidt \cite{ESS02}. 
\begin{theorem}\label{th-subspace}
Let $\Gamma\subset \bbc^\times$ be a a multiplicative group of rank $r$. Let $m\geq 2$. For any fixed $a_0,a_1,\ldots,a_m\in \bbc\backslash \{0\}$ the number of solutions to
\[a_0=a_1z_1+\cdots+a_mz_m\textrm{ with }z_i\in \Gamma,\]
such that no non-empty subsum of the right-hand side equals $0$, is at most
\[m^{O(m^4(m+r))}.\]
\end{theorem}

This immediately imposes strong restrictions on the potential additive structure of subsets of multiplicative groups of bounded rank -- for example, if $A$ is a finite subset of a multiplicative group of rank $r$ then Theorem~\ref{th-subspace}, applied with $a_0=x$, $a_1=1$, and $a_2=-1$, implies that for any $x\neq 0$
\[1_A\circ 1_A(x) \ll_r 1,\]
and hence in particular $E(A) \ll_r \abs{A}^2$ and $\abs{A+A}\gg_r \abs{A}^2$.

A large part of the power of \cite{HRSZ25} arises from finding ways to apply Theorem~\ref{th-subspace} to efficiently bound the additive structure, not only of subsets of multiplicative groups of small rank, but of subsets of the union of `few' cosets of such a group. We will give various improved forms of such bounds, giving new bounds for both $1_A\circ 1_A(x)$ for any fixed $x\neq 0$ and $E_{2m}(A)$ for any $m\geq 1$.

\subsection{A pointwise bound on $1_A\circ 1_A$}
An immediate consequence of Theorem~\ref{th-subspace} is that if $A$ is $M$-covered by a rank $r$ multiplicative group then, for all $x\neq 0$, $1_A\circ 1_A(x) \ll 2^{O(r)}M^2$. Indeed, if $A\subseteq \Gamma\cdot B$, then fixing the values of $b_1$ and $b_2$ at a cost of $M^2$ the number of choices of $\gamma_1,\gamma_2\in \Gamma$ with $\gamma_1b_1-\gamma_2b_2=x$ is at most $2^{O(r)}$. This bound can be improved (for reasonable choices of $r$ and $M$) to $M^{1+o(1)}$, by generalising an argument of Roche-Newton and Zhelezov \cite{RoZh15}.
\begin{lemma}\label{lem-point}
If $r,M\geq 1$ and $\delta\in (0,1)$ are such that 
\[\delta \log M \geq \max(r, (\log M)^{1/6})\]
and $A\subset \bbc^\times$ is $M$-covered by a rank $r$ multiplicative group then, for any $x\neq 0$, 
\[1_A\circ 1_A(x)\ll M^{1+O(\delta^{1/5})}.\]
\end{lemma}
One should not take the constants $1/6$ and $1/5$ too seriously; an examination of the proof shows that they can be improved slightly at notational expense.
\begin{proof}
Let $\Gamma$ be a rank $r$ multiplicative group such that $A\subseteq \Gamma\cdot B$, where $B\subseteq \bbc^\times$ is a finite set of size $\abs{B}=M$. Without loss of generality $B$ can be taken to be a set of coset representatives -- that is, $\Gamma\cdot b_1$ and $\Gamma\cdot b_2$ are disjoint for $b_1\neq b_2\in B$. We may also assume,without loss of generality, that $-1\in \Gamma$.

It suffices to bound the number of $b_1,b_2\in B$ for which there exist $\gamma_1,\gamma_2\in \Gamma$ with
\[\gamma_1b_1-\gamma_2b_2=x,\]
since once $b_1,b_2$ are fixed there are at most $2^{O(r)}$ choices for $\gamma_1,\gamma_2\in \Gamma$ by Theorem~\ref{th-subspace}.

Let $G$ be a directed graph with vertex set $B$ such that there is an edge $b_1\to b_2$ if and only if there exist $\gamma_1,\gamma_2\in \Gamma$ with $\gamma_1b_1-\gamma_2b_2=x$, so that we need to bound the number of edges in $G$; let this be denoted by $dM$.

By repeatedly removing vertices of out-degree $<d/2$ we can find a subgraph, say $G'$, which contains $\geq dM/2$ many edges, and $G'$ has minimum out-degree $\geq d/2$. For each edge $b_i\to b_j$ we fix some pair $\gamma_i,\mu_j\in \Gamma$ such that $\gamma_ib_i-\mu_jb_j=x$.

Given a path $b_0\cdots b_l$ in $G'$, with $\gamma_ib_i-\mu_{i+1}b_{i+1}=x$ for $0\leq i<l$, writing $\lambda_i=\mu_i\gamma_i^{-1}$, by telescoping the sum,
\begin{align*}
\gamma_0b_0-\lambda_1\cdots \lambda_{l-1}\mu_lb_l
&= \gamma_0b_0-\mu_1b_1+\lambda_1(\gamma_1b_1-\mu_2b_2)+\cdots\\
&\quad+\lambda_1\cdots\lambda_{l-1}(\gamma_{l-1}b_{l-1}-\mu_lb_l)\\
&=x(1+\lambda_1+\cdots+\lambda_1\cdots \lambda_{l-1}).
\end{align*}
In other words, each path of length $l$ between two fixed vertices $b_0$ and $b_l$ yields a solution to the equation
\begin{equation}\label{eq-path}
1+z_1+\cdots+z_{l-1}+z_l(b_0x^{-1})+z_{l+1}(b_lx^{-1})=0,
\end{equation}
where all $z_i\in \Gamma$, corresponding to $z_i=\lambda_1\cdots \lambda_i$ for $1\leq i<l$, $z_l=-\gamma_0$, and $z_{l+1}=\lambda_1\cdots \lambda_{l-1}\mu_l$.

Moreover, any two distinct such paths between the same endpoints $b_0$ and $b_l$ yield distinct $z_1,\ldots,z_{l+1}$. Indeed, fixing all $z_i$ fixes $\lambda_i$ for $1\leq i<l$, and $\gamma_0$ and $\mu_l$. Since $\mu_1b_1=x-\gamma_0b_0$ is known, and since $b_1\in B$ which is a set of coset representatives, we can recover the value of both $\mu_1$ and $b_1$. Since we know $\lambda_1$ we also know $\gamma_1$. In general, once $b_i,\gamma_i,\mu_i$ are known for all $0\leq i\leq j$, the value of $\mu_{j+1}b_{j+1}=x-\gamma_ib_i$ is also known, whence $\mu_{j+1},b_{j+1},\gamma_{j+1}$ are also known. In particular fixing all $z_i$ fixes all $b_0,\ldots,b_l$ as claimed.

We would now like to apply Theorem~\ref{th-subspace} to bound the number of such $z_i$ (and hence the number of such paths), but we must take care that we are only bounding the number of non-degenerate solutions. We therefore call a path $b_0\cdots b_l$ non-degenerate if no proper subsum of the left-hand side of \eqref{eq-path} vanishes.

This is achieved inductively -- for $l=1$ this amounts to ensuring  that no proper subsum of 
\[x-\gamma_0b_0+\mu_1b_1\]
vanishes, which is trivial. In particular every path of length $1$ is non-degenerate.

In general, we claim that either $d\ll 2^l$ or any non-degenerate path $b_0\cdots b_l$ of length $l$ can be extended to a non-degenerate path of length $l+1$ in at least $d/4$ many ways. For $b_0\cdots b_{l+1}$ to be non-degenerate requires no proper subsum of
\[\brac{1+\lambda_1+\cdots+\lambda_1\cdots\lambda_{l-1}-\gamma_0b_0x^{-1}}+\lambda_1\cdots\lambda_l+\lambda_1\cdots\lambda_l\mu_{l+1}b_{l+1}x^{-1}\]
vanishes. By non-degeneracy of $b_0\cdots b_l$ the first bracketed sum (and any subsum) cannot vanish. Let $\Sigma$ be the set of values of all subsums of the first bracketed sum. To ensure that the extension by $b_{l+1}$ is still non-degenerate it therefore suffices to ensure that
\[\lambda_l\not\in -(\lambda_1\cdots \lambda_{l-1})^{-1}\Sigma.\]
Recalling that $\lambda_l=\mu_l\gamma_l^{-1}$ (and all $\gamma_i$ for $1\leq i<l$ and $\mu_j$ for $1\leq j\leq l$ are determined by $b_0\cdots b_l$) this amounts to a set of at most $2^{l}$ many forbidden values for $\gamma_l$.

Furthermore, note that fixing a value of $\gamma_l$ fixes $\mu_{l+1}b_{l+1}=\gamma_lb_l-x$ and hence also fixes $b_{l+1}$. In other words, for a fixed non-degenerate path $b_0\cdots b_l$ there are at most $2^l$ possible $b_{l+1}$ such that $b_0\cdots b_{l+1}$ forms a degenerate path. Since the minimum degree of $G'$ is $d/2$, either $d\leq 2^{l+2}$, or there are at least $d/4$ extensions to a non-degenerate path as claimed.

It follows that either $d\ll 2^l$ or there are at least $(d/4)^l$ many non-degenerate paths beginning at any fixed $b_0\in B'$. We can fix the endpoint $b_l$ losing only a factor of $\abs{B}$. By the discussion above the number of non-degenerate paths with fixed endpoints $b_0$ and $b_l$ is $\ll l^{O(l^4(l+r))}$, and hence
\[(d/4)^l\abs{B}^{-1}\ll l^{O(l^4(l+r))},\]
so that
\[d \ll \abs{B}^{1/l} l^{O(l^3(l+r))}.\]
Choosing $l=\lfloor \delta^{-1/5}\rfloor$ yields the result.
\end{proof}

\subsection{Bounds on additive energies}

Recall that the additive energy $E(A)$ counts the number of solutions to $a_1+a_2-a_3=a_4$ with $a_i\in A$. If $A$ is $M$-covered by a rank $r$ multiplicative group, say $A\subseteq \Gamma\cdot B$ with $\abs{B}\leq M$, then Theorem~\ref{th-subspace} implies $E(A)\ll_r \abs{A}^2+\abs{A}M^3$. Indeed, after spending $\abs{A}$ to fix $a_4$ and $M^3$ to fix $b_1,b_2,b_3$ there are $O_r(1)$ many choices for $\gamma_1,\gamma_2,\gamma_3$ with
\[\gamma_1b_1+\gamma_2b_2-\gamma_3b_3=a_4.\]
This almost works, except that we need to be sure that no proper subsum of the left-hand side vanishes. Provided $0\not\in A$, it is easy to see that such degenerate solutions contribute $O(\abs{A}^2)$ to the additive energy, whence
\[E(A) \ll_r \abs{A}^2+\abs{A}M^3\]
as claimed. In this section we will prove an improvement over this, establishing that, in fact,
\[E(A) \ll_r \abs{A}^2+\abs{A}M^2.\]
The key observation is that the obvious generalisation of the preceding argument gives non-trivial bounds for higher additive energies $E_{2m}(A)$ for all $m\geq 1$, namely $E_{2m}(A) \ll \abs{A}^m+\abs{A}M^{2m-1}$. Furthermore, by H\"{o}lder's inequality we can efficiently bound 
\[E(A)\leq \abs{A}^{\frac{m-2}{m-1}}E_{2m}(A)^{\frac{1}{m-1}}.\] 
In particular,
\[E(A) \ll \abs{A}^2+\abs{A}M^{\frac{2m-1}{m-1}}.\]
Taking $m\to \infty$ yields the claimed result. In the remainder of this section we make this sketch precise. 

\begin{lemma}\label{lem-energies}
Let $A$ be a finite set in an abelian group. Let $k,r,n$ be integers such that $k\geq n/2\geq r\geq 1$. For any $x$ and $\epsilon_1,\ldots,\epsilon_n\in \{-1,1\}$ the number of solutions to
\[x=\epsilon_1a_1+\cdots+\epsilon_na_n\]
with $a_i\in A$ for $1\leq i\leq n$ is at most
\[E_{2r}(A)^{\frac{2k-n}{2k-2r}}E_{2k}(A)^{\frac{n-2r}{2k-2r}}.\]
\end{lemma}
\begin{proof}
Passing to a Freiman-isomorphic model if necessary (see \cite[Chapter 5]{TaVu06}) we may assume that the ambient group is finite. By orthogonality, the count to be estimated is equal to
\[\bbe_{\gamma}  \gamma(-x)\prod_{1\leq i\leq n}\widehat{1_A}(\epsilon_i\gamma) ,\]
where the expectation is over the dual group, and $\widehat{1_A}(\gamma)=\sum_{n\in A}\gamma(n)$. By the triangle inequality and then H\"{o}lder's inequality this is at most
\[\bbe_\gamma \Abs{\widehat{1_A}(\gamma)}^n\leq \brac{\bbe_\gamma \Abs{\widehat{1_A}(\gamma)}^{2r}}^{\frac{2k-n}{2k-2r}}\brac{\bbe_\gamma \Abs{\widehat{1_A}(\gamma)}^{2k}}^{\frac{n-2r}{2k-2r}}.\]
The claim now follows since, for example, orthogonality implies $\bbe_\gamma \Abs{\widehat{1_A}(\gamma)}^{2r}=E_{2r}(A)$.
\end{proof}

\begin{lemma}\label{lem-energy}
If $A\subset \bbc$ is $M$-covered by a rank $r$ multiplicative group then, for any $m\geq k\geq 1$,
\[E_{2k}(A)\leq 2^{O(km)}\abs{A}^{k}+m^{O(km^3(m+r))}\abs{A}M^{2k-2+\frac{k-1}{m-1}}.\]
\end{lemma}
In particular, note that taking $k=2$ and letting $m\to \infty$ slowly implies
\[E_4(A) \leq \abs{A}^{2+o(1)}+\abs{A}^{1+o(r)}M^2.\]
\begin{proof}
Let $2\leq \ell <2m$. By Lemma~\ref{lem-energies}, applied with $r=1$ and $k=m$, the contribution to $E_{2m}(A)$ from those $2m$-tuples in which a subset of $\ell$ summands sums to $0$ is at most
\[\binom{2m}{\ell}\brac{ \abs{A}^{\frac{2m-\ell}{2m-2}}E_{2m}(A)^{\frac{\ell-2}{2m-2}}}\brac{ \abs{A}^{\frac{2m-(2m-\ell)}{2m-2}}E_{2m}(A)^{\frac{2m-\ell-2}{2m-2}}}\]
\[\leq \binom{2m}{\ell}\abs{A}^{\frac{m}{m-1}}E_{2m}(A)^{\frac{m-2}{m-1}}.\]
Let $E_{2m}^*(A)$ count the number of tuples 
\[a_1+\cdots-a_{2m}=0\]
such that no non-empty subsum equals $0$. Summing the preceding inequality over all $2\leq \ell\leq 2m-2$ (and noting that since $0\not\in A$ no subset of $1$ or $2m-1$ summands can sum to $0$)
\[E_{2m}(A) \leq E_{2m}^*(A)+2^{O(m)}\abs{A}^{\frac{m}{m-1}}E_{2m}(A)^{\frac{m-2}{m-1}}\]
and hence
\[E_{2m}(A) \ll E_{2m}^*(A)+2^{O(m^2)}\abs{A}^m.\]
Let $\Gamma$ be a multiplicative group of rank $r$ and let $B$ be a set of size $M$ such that $A\subset\Gamma\cdot B$. The count $E_{2m}^*(A)$ is at most the number of solutions to
\[x_1b_1+\cdots+x_{2m-1}b_{2m-1}=a_{2m}\]
with no subsum being zero, where $a_{2m}\in A$, $b_i\in B$, and $x_i\in \Gamma$. There are at most $\abs{A}\abs{B}^{2m-1}$ choices for $b_1,\ldots,b_{2m-1},a_{2m}$, after which, by Theorem~\ref{th-subspace}, there are $\leq m^{O(m^4(m+r))}$ choices for the $x_i$. It follows that
\[E_{2m}(A) \ll 2^{O(m^2)}\abs{A}^m+ m^{O(m^4(m+r))}\abs{A}\abs{B}^{2m-1}.\]
We now amplify this by applying Lemma~\ref{lem-energies} (with $n=2k$, $x=0$, and $r=1$), so that, for any $1\leq k\leq m$,
\[E_{2k}(A)\leq \abs{A}^{\frac{m-k}{m-1}}E_{2m}(A)^{\frac{k-1}{m-1}},\]
and the conclusion follows.
\end{proof}

\section{Lower bounds on the sumset}\label{sec-add2}

An immediate implication of the energy bounds of Lemma~\ref{lem-energy} is the following, which (when coupled with the results of Section~\ref{sec-comb}) implies the main result of \cite{HRSZ25}.

\begin{theorem}\label{th-olden}
Let $r,M\geq 1$ and $\delta\in (0,1)$ be such that
\[\delta \log M \geq \max(r, (\log M)^{1/6}).\]
If $A\subset \bbc$ is $M$-covered by a rank $r$ multiplicative group then
\[E(A)\ll 2^{O(\delta^{-1/5})}\abs{A}^2+\abs{A}M^{2+O(\delta^{1/5})}.\]
In particular, if $\abs{A+A}=K\abs{A}$, then
\[\max(K,M)\geq \abs{A}^{\frac{2}{3}-O(\delta^{-1/5})}.\]
\end{theorem}
\begin{proof}
The upper bound on $E(A)=E_4(A)$ follows from applying Lemma~\ref{lem-energy} with $k=2$ and $m=\lfloor \delta^{-1/5}\rfloor$. The second bound is a consequence of the inequality \eqref{eq-holderenergy}, which yields
\[\abs{A+A}\geq \frac{\abs{A}^4}{E(A)}\gg \min\brac{  2^{-O(\delta^{-1/5})}\abs{A}^2, \abs{A}^{3}M^{-2-O(\delta^{1/5})}}.\]
\end{proof}

Recalling that, in the application to sets of integers with $\leq k$ prime factors, Proposition~\ref{prop-small} allows us to take $M\ll_k \abs{AA}/\abs{A}$, this recovers the bound
\[\max(\abs{AA},\abs{A+A})\geq \abs{A}^{5/3-o(1)}\]
of \cite{HRSZ25}.

In the remainder of this section we will use combinatorial arguments and the non-trivial upper bounds for $E_4(A)$, $E_8(A)$, and $\max_{x\neq 0}1_A\circ 1_A(x)$ from the last section to go beyond this exponent of $5/3$ -- albeit only for the size of $A+A$ (or $A-A$), rather than the additive energy. We remind the reader that, bearing in mind the example of Balog and Wooley, the bound on the additive energy in Theorem~\ref{th-olden} cannot be improved.

The new ingredient is an inequality from the theory of higher additive energies, which has been extensively developed in a number of papers of Shkredov and Schoen-Shkredov. This theory explores the relationship between the higher additive energies $E_{2m}(A)$ and higher moments of the convolution. The following inequality will suffice for our purposes. It is a variant of a result given in various special forms in a number of papers of Schoen and Shkredov -- see, for example, \cite[Remark 40]{Sh26} and \cite[Lemma 3]{ScSh13}. 

We give a short self-contained proof of a general form of this inequality (we will only require the $k=2$ case of this for our application, but prove the more general case here for completeness).
\begin{lemma}\label{lem-shkredov}
For any finite sets $A$, $B$, and $C$, and any $k\geq 1$
\[\brac{\sum_{c\in C}1_A\ast 1_B(c)}^{4k}\ll \abs{A}^{2k}\abs{B}^{2k}E_{2k}(C)\brac{E_{2k}(C)+\sum_{x\neq 0}1_A\circ 1_A(x)^{k}1_B\circ 1_B(x)^k},\]
where the implicit constant is absolute.
\end{lemma}
\begin{proof}
Consider the bipartite graph $G$ with vertex set $A\times B$ in which $a$ and $b$ are joined by an edge if $a+b\in C$, so that the number of edges in $G$ is precisely $\sum_{c\in C}1_A\ast 1_B(c)=\delta \abs{A}\abs{B}$, say.

Let $V_{2k}$ count the number of cycles of length $2k$ in $G$. One one hand, $V_{2k}\geq \delta^{2k}\abs{A}^k\abs{B}^k$ -- this is a well-known fact in graph theory, and is a special case of Sidorenko's conjecture. The particular case of even cycles was proved by Sidorenko \cite{Si91}. The identity
\[(a_1+b_1)+(a_2+b_2)+(a_k+b_k)=(b_1+a_2)+\cdots+(b_k+a_1)\]
implies that, if $a_1b_1\cdots a_kb_k$ is a cycle, then letting $c_i=a_i+b_i$ and $d_i=b_i+a_{i+1}$ (where $a_{k+1}=a_1$) we have $c_i,d_i\in C$ and
\[c_1+\cdots+c_k=d_1+\cdots+d_k.\]
In other words, the number of cycles of length $2k$ is equal to
\[V_{2k}=\sum_{\substack{c_1,\ldots,c_k,d_1,\ldots,d_k\in C\\ c_1+\cdots+c_k=d_1+\cdots+d_k}}\sum_{a_1,\ldots,a_k\in A}\sum_{b_1,\ldots,b_k\in B}\prod_{i=1}^k 1_{a_i+b_i=c_i}1_{b_i+a_{i+1}=d_i}.\]
By the Cauchy-Schwarz inequality it follows that $V_{2k}^2$ is bounded above by
\[E_{2k}(C)\brac{ \sum_{\substack{a_1,\ldots,a_k'\in A\\b_1,\ldots,b_k'\in B}}\sum_{\substack{c_1,\ldots,d_k\in C\\ c_1+\cdots+c_k=d_1+\cdots+d_k}} \prod_{i=1}^k 1_{a_i+b_i=c_i=a_i'+b_i'}1_{b_i+a_{i+1}=d_i=b_i'+a_{i+1}}}.\]
In the bracketed sum, if $a_i=a_i'$ or $b_i=b_i'$ for any $1\leq i\leq k$ then $a_i=a_i'$ and $b_i=b_i'$ for all $1\leq i\leq k$, and hence the bracketed expression equals $V_{2k}$. Otherwise it is at most
\[\sum_{a_1,\ldots,a_k'\in A}\sum_{b_1,\ldots,b_k'\in B} \prod_{i=1}^k  1_{0\neq a_i-a_i'=b_i'-b_i=a_{i+1}-a_{i+1}'}=\sum_{x\neq 0}1_A\circ 1_A(x)^k1_B\circ 1_B(x)^k.\]
That is,
\[V_{2k}^2\leq E_{2k}(C)\brac{V_{2k}+\sum_{x\neq 0}1_A\circ 1_A(x)^k1_B\circ 1_B(x)^k}\]
which after rearranging and recalling $V_{2k}\geq \delta^{2k}\abs{A}^k\abs{B}^k$ yields the result.
\end{proof}

Taking $C$ to be the set of popular sums or differences implies the following.
\begin{lemma}\label{lem-pop}
If either $\abs{A+A}\leq K\abs{A}$ or $\abs{A-A}\leq K\abs{A}$ then, for any $\ell\geq k\geq 1$,
\[\abs{A}^{6k-3}\ll_k K^{2k-1}E_{4\ell}(A)^{\frac{k-1}{\ell-1}}\brac{ K^{2k-1}\abs{A}^{3-2k}E_{4\ell}(A)^{\frac{k-1}{\ell-1}}+ \sum_{x\neq 0}1_A\circ 1_A(x)^{2k}},\]
where the implied constant depends on $k$ only.
\end{lemma}
\begin{proof}
We first prove the result for sums by taking $B=A$. Let $C=\{ x : 1_A\ast 1_A(x)\geq \frac{1}{2K}\abs{A}\}$, so that $\sum_{c\in C}1_A\ast 1_A(c) \geq \frac{1}{2}\abs{A}^2$. By Lemma~\ref{lem-shkredov} it follows that
\[\abs{A}^{4k}\ll_k E_{2k}(C)\brac{E_{2k}(C)+\sum_{x\neq 0} 1_A\circ 1_A(x)^{2k}}.\]
By Lemma~\ref{lem-energies}, for any $\ell \geq k$,
\[E_{2k}(C) \leq \abs{C}^{\frac{\ell-k}{\ell-1}}E_{2\ell}(C)^{\frac{k-1}{\ell-1}}.\]
It follows that, since $\abs{C}\leq K\abs{A}$ and, noting $(2K)^{-1}\abs{A}1_C\leq 1_A\ast 1_A$,
\[E_{2\ell}(C) \leq (2K)^{2\ell}\abs{A}^{-2\ell}E_{4\ell}(A),\]
we have
\[E_{2k}(C) \ll_k K^{2k-1+\frac{k-1}{\ell-1}}\abs{A}^{3-2k-3\frac{k-1}{\ell-1}}E_{4\ell}(A)^{\frac{k-1}{\ell-1}}.\]
and hence (taking advantage of $K\leq \abs{A}^3$ to simplify exponents slightly)
\[\abs{A}^{4k}\ll_k K^{4k-2}\abs{A}^{6-4k}E_{4\ell}(A)^{2\frac{k-1}{\ell-1}}+ K^{2k-1}\abs{A}^{3-2k}E_{4\ell}(A)^{\frac{k-1}{\ell-1}}\sum_{x\neq 0}1_A\circ 1_A(x)^{2k}\]
which yields the result.

The argument for differences is nearly identical. This time we take $B=-A$ and $C=\{ x : 1_A\circ 1_A(x)\geq \frac{1}{2K}\abs{A}\}$ and repeat the above. The simple observations that $1_A \circ 1_A(x) = 1_{-A} \circ 1_{-A}(x)$ and that 
\[ E_{2\ell}(C) \leq (2K)^{2\ell}\abs{A}^{-2\ell}E_{4\ell}(A)\]
ensure that all steps taken for sums are valid.
\end{proof}

We may now use Lemma~\ref{lem-pop} together with the bounds from Section~\ref{sec-add} to obtain an improved lower bound on both $\abs{A+A}$ and $\abs{A-A}$ for sets $A$ which are efficiently covered by a low rank multiplicative group.

\begin{theorem}\label{th-growth}
Let $r,M\geq 1$ and $\delta\in (0,1)$ be such that
\[\delta \log M \geq \max(r,(\log M)^{1/6}).\]
If $A\subset \bbc$ is $M$-covered by a rank $r$ multiplicative group and 
\[\min(\abs{A+A},\abs{A-A})=K\abs{A}\]
then
\[\max(K,M)\gg \abs{A}^{5/7-O(\delta^{1/5})} .\]
\end{theorem}
\begin{proof}
By Lemma~\ref{lem-energy}
\[E_4(A) \ll 2^{O(m)}\abs{A}^2+m^{O(m^3(r+m))}\abs{A}M^{2+\frac{2}{m-1}}.\] 
We apply this with $m=\lfloor \delta^{-1/5}\rfloor$, so that
\[E_4(A) \lesssim \abs{A}^2+\abs{A}M^2,\]
where $\ls$ hides losses polynomial in $\abs{A}^{\delta^{1/5}}$. By Lemma~\ref{lem-point} it follows that
\[\sum_{x\neq 0} 1_A\circ 1_A(x)^4\lesssim M^2\brac{\abs{A}^2+\abs{A}M^2}.\]
Applying Lemma~\ref{lem-pop} with $k=2$ therefore implies, for any $\ell\geq 2$,
\[\abs{A}^{9}\ls K^{3}E_{4\ell}(A)^{\frac{1}{\ell-1}}\brac{K^3\abs{A}^{-1}E_{4\ell}(A)^{\frac{1}{\ell-1}}+M^2\abs{A}^2+M^4\abs{A}}.\]
By Lemma~\ref{lem-energy} again 
\[E_{4\ell}(A)^{\frac{1}{\ell-1}} \ls \abs{A}^{2+O(1/\ell)}+\abs{A}^{O(1/\ell)}M^{4}\ls \abs{A}^2+M^4\]
taking $\ell$ sufficiently large ($\asymp \log \abs{A}$ say). It follows that 
\[\abs{A}^{9}\ls K^{3}(\abs{A}^2+M^4)\brac{K^3\abs{A}^{-1}(\abs{A}^2+M^4)+M^2\abs{A}^2+M^4\abs{A}}.\]
Simplifying the right-hand side yields
\[\abs{A}^9\ls K^6\abs{A}^3+K^6M^4\abs{A}+K^3M^2\abs{A}^4+K^3M^4\abs{A}^3\]
\[ +K^6M^8\abs{A}^{-1}+K^3M^6\abs{A}^2+K^3M^8\abs{A}\]
which implies
\[\max(K,M) \gtrsim \abs{A}^{5/7}\]
as claimed. 
\end{proof}
We highlight here that in the proof above we showed that, if $C$ is the set of popular sums or differences used in the proof of Theorem\ref{th-growth}, then
\[ E_4(C) \lesssim K^3 (M^4 + \abs{A}^2) \abs{A}^{-1}.\]

We may now prove the full quantitative version of Theorem~\ref{th-main}.

\begin{theorem}\label{th-mainprecise}
Let $A\subset \bbq\backslash \{0\}$ be a finite set such that $\omega(n)\leq k$ for all $n\in A$. If $\delta\in (0,1)$ is such that
\[\delta \log \abs{A} \geq \max(k,(\log \abs{A})^{1/6})\]
then
\[\max(\abs{A+A},\abs{AA})\geq O(k)^{-k}\abs{A}^{12/7-O(\delta^{1/5})}\]
and
\[\max(\abs{A-A},\abs{AA})\geq O(k)^{-k}\abs{A}^{12/7-O(\delta^{1/5})} .\] 
\end{theorem}
Note in particular that we obtain a lower bound of $\abs{A}^{12/7-o(1)}$ provided $k=o(\frac{\log \abs{A}}{\log\log \abs{A}})$. 
\begin{proof}
Let $K\abs{A}=\max(\abs{A+A},\abs{AA})$. (The proof for the case of $\max(\abs{A-A},\abs{AA})$ is identical.) By Proposition~\ref{prop-small} there is a set $A'\subseteq A$ of size $\abs{A'}\gg (2k)^{-k}\abs{A}$ which is $M$-covered by a rank $r$ multiplicative group, with $M\ll 4^kK$ and $r\ll k$. After dilating $\delta$ by some some small absolute constant if necessary, we may assume that
\[\delta \log M \geq \max(r,(\log M)^{1/6}),\]
so that we can apply Theorem~\ref{th-growth} to deduce that, if $\abs{A'+A'}= K'\abs{A'}$, 
\[\max(K',M) \gg \abs{A}^{5/7-O(\delta^{1/5})},\]
which yields the desired conclusion.
\end{proof}

\section{Many-fold sums and products}

Next, we prove our result on many sums and products, Theorem~\ref{th-main2}. This is, more or less, a direct consequence of the bounds on additive energies in Lemma~\ref{lem-energy}. We first prove the most general quantitative statement possible, and then specialise to obtain Theorem~\ref{th-main2}.

\begin{theorem}\label{th-main2precise}
Let $A\subset \bbq\backslash\{0\}$ be a finite set such that $\omega(n)\leq k$ for all $n\in A$. If $\delta\in (0,1)$ is such that
\[\delta \log \abs{A} \geq \max(k,(\log \abs{A})^{1/6})\]
then for any $1/2\leq c\leq 1$ and $m\geq 2$ either
\[\Abs{A^{(m)}}> \abs{A}^{c(m-1)+1}\]
or there is $A'\subseteq A$ of size $\abs{A'}\gg (2mk)^{-k}\abs{A}$ such that 
\[E_{2m}(A')\ll \abs{A'}^{1+2c(m-1)+O(m\delta^{1/5})}.\]
In particular, either
\[\Abs{A^{(m)}} > \abs{A}^{cm+1-c}\textrm{ or }\Abs{mA}\gg (mk)^{-O(mk)} \Abs{A}^{2(1-c)m+2c-1-O(m\delta^{1/5})}.\]
\end{theorem}
\begin{proof}
Let $m\geq 2$ and suppose that $\Abs{A^{(m)}}\leq \abs{A}^{c(m-1)+1}$. By the pigeonhole principle there exists some $1\leq i<m$ such that, with $B=A^{(i)}$, we have $\abs{AB}\leq \abs{A}^c\abs{B}$. Furthermore, clearly $\omega(n)\leq mk$ for all $n\in B$. It follows by Proposition~\ref{prop-small} that there is $A'\subseteq A$ of size $\abs{A'}\gg (2mk)^{-k}\abs{A}$ which is $M$-covered by a rank $r$ multiplicative group, where $r\ll k$ and $M\ll 2^{O(mk)}\abs{A}^c$.

We now apply Lemma~\ref{lem-energy} with $k$ replaced by $m$ and $m$ replaced by $\lfloor\delta^{-1/5}\rfloor$ to deduce that
\[E_{2m}(A') \leq 2^{O(m\delta^{-1/5})}\abs{A'}^m+\abs{A'}^{1+2c(m-1)+O(m\delta^{1/5})}.\]
Since $m\leq 1+2c(m-1)$ this simplifies to give the inequality in the theorem statement.
\end{proof}

Choosing $c=2/3$ we deduce the following precise formulation of Theorem~\ref{th-main2} (note that when $m=2$ this recovers the bound of \cite{HRSZ25} as a special case). 
\begin{theorem}
Let $A\subset \bbq\backslash\{0\}$ be a finite set such that $\omega(n)\leq k$ for all $n\in A$. If $\delta\in (0,1)$ is such that
\[\delta \log \abs{A} \geq \max(k,(\log \abs{A})^{1/6})\]
then for any $m\geq 2$
\[\max(\Abs{mA},\Abs{A^{(m)}}) \gg (mk)^{-O(mk)}\abs{A}^{\frac{2}{3}m+\frac{1}{3}-O(m\delta^{1/5})}.\]
\end{theorem}

\section{On the size of $\abs{A+AA}$}
Finally, we prove Theorem~\ref{th-main3}, a near-optimal lower bound for $\abs{A+AA}$ for sets of rationals with few prime divisors. We require the following result, an asymmetric energy bound, which is similar to Lemma~\ref{lem-energies}.
\begin{lemma}\label{lem-energiesasymm}
Let $A$ and $B$ be finite sets in any abelian group. For any integers $m,n\geq 2$
\[E(A,B) \leq E_{2m}(A)^{\frac{1}{m}}E_{2n}(B)^{\frac{1}{m(n-1)}}\abs{B}^{1-\frac{n}{m(n-1)}}.\]
\end{lemma}
\begin{proof}
Passing to a Freiman-isomorphic model if necessary (see \cite[Chapter 5]{TaVu06}) we may assume that the ambient group is finite. By orthogonality, the count to be estimated is equal to
\[\bbe_{\gamma}  \Abs{\widehat{1_A}(\gamma)}^2\Abs{\widehat{1_B}(\gamma)}^2,\]
where the expectation is over the dual group, and, for example,  $\widehat{1_A}(\gamma)=\sum_{n\in A}\gamma(n)$. By H\"{o}lder's inequality this is at most
\begin{align*}
&\leq \left(\bbe_{\gamma}|\widehat{1_{A}}(\gamma)|^{2m}\right)^{\frac{1}{m}}\left(\bbe_{\gamma}|\widehat{1_{B}}(\gamma)|^{2n}\right)^{\frac{1}{m(n-1)}}\left(\bbe_{\gamma}|\widehat{1_{B}}(\gamma)|^2\right)^{1-\frac{n}{m(n-1)}}\\
            &=E_{2m}(A)^{\frac{1}{m}}E_{2n}(B)^{\frac{1}{m(n-1)}} |B|^{1-\frac{n}{m(n-1)}}
\end{align*}
as required, where once again we have used orthogonality in the final equality.  
\end{proof}

Using this coupled with our existing higher energy estimates yields the following.
\begin{theorem}
Let $A\subset \bbq\backslash\{0\}$ be a finite set such that $\omega(n)\leq k$ for all $n\in A$. If $\delta\in (0,1)$ is such that
\[\delta \log \abs{A} \geq \max(k,(\log \abs{A})^{1/6})\]
then, for any finite set $B$,
\[\abs{A+B}\geq O(k)^{-k}(\abs{A}\abs{B})^{1-O(\delta^{-1/5})}\min\brac{1, \frac{\abs{A}^3}{\abs{AA}^2}}.\]
\end{theorem}
\begin{proof}
Let $K=\abs{AA}/\abs{A}$. By Proposition~\ref{prop-small} there is a set $A'\subseteq A$ of size $\abs{A'}\gg (2k)^{-k}\abs{A}$ which is $M$-covered by a rank $r$ multiplicative group, with $M\ll 4^kK$ and $r\ll k$. For any set $B$, applying Lemma~\ref{lem-energiesasymm} with $n=2$ and using the trivial bound of $E(B)\leq \abs{B}^3$ together with Lemma~\ref{lem-energy} we have
\[E(A',B)\ll m^{O(m^3(m+r))}\brac{\abs{A}^m+\abs{A}M^{2m-1}}^{\frac{1}{m}}\abs{B}^{1+1/m}.\]
In particular, taking $m=\lfloor \delta^{-1/5}\rfloor$ as in the proof of Theorem~\ref{th-growth} we have
\[E(A',B)\ls (\abs{A}+M^2)\abs{B}\]
where $\ls$ hides losses polynomial in $(\abs{A}\abs{B})^{\delta^{1/5}}$. The claim now follows from the Cauchy-Schwarz inequality, which implies
\[\abs{A'+B}\geq \frac{\abs{A'}^2\abs{B}^2}{E(A',B)}.\]
\end{proof}

Taking $B=AA$ immediately yields the following corollary.

\begin{corollary}\label{cor-main3precise}
Let $A\subset \bbq\backslash\{0\}$ be a finite set such that $\omega(n)\leq k$ for all $n\in A$. If $\delta\in (0,1)$ is such that
\[\delta \log \abs{A} \geq \max(k,(\log \abs{A})^{1/6})\]
then
\[\abs{A+AA}\geq O(k)^{-k}\abs{A}^{-O(\delta^{-1/5})}\min\brac{\abs{A}\abs{AA}, \frac{\abs{A}^4}{\abs{AA}}}.\]
In particular, 
\[\abs{A+AA}\geq O(k)^{-k}\abs{A}^{2-O(\delta^{-1/5})}.\]
\end{corollary}
Note in particular that we obtain a lower bound of $\abs{A}^{2-o(1)}$ provided $k=o(\frac{\log \abs{A}}{\log\log \abs{A}})$. Finally, we note that if $\abs{A+AA}\leq M\abs{A}^2$, say, then either
\[\abs{AA}\ls M\abs{A}\]
or 
\[\abs{AA}\gtrsim M^{-1}\abs{A}^2\]
(where $\ls$ hides losses polynomial in $k^{-k}\abs{A}^{-\delta^{1/5}}$). In the former case, the inequality from the end of the proof of Theorem~\ref{th-growth} yields
\[\abs{A}^9\ls M^{O(1)}\brac{K^6\abs{A}+K^3\abs{A}^2+K^6\abs{A}^3+K^3\abs{A}^4},\]
and hence $K \geq M^{-O(1)}\abs{A}$, and we deduce the following.
\begin{corollary}
Let $A\subset \bbq\backslash\{0\}$ be a finite set such that $\omega(n)\leq k$ for all $n\in A$. If $\delta\in (0,1)$ is such that
\[\delta \log \abs{A} \geq \max(k,(\log \abs{A})^{1/6})\]
and 
\[\abs{A+AA}\leq M\abs{A}^2\]
then
\[\max(\abs{A+A},\abs{AA})\geq O(k)^{-k}M^{-O(1)}\abs{A}^{2-O(\delta^{-1/5})}.\]
\end{corollary}

\bibliographystyle{plain}
\bibliography{FewPrimesSP} 
\end{document}